\documentclass[12pt,oneside]{amsart}
\usepackage{amsthm}
\usepackage{amsmath}
\usepackage{amssymb}
\usepackage{tikz}
\usepackage{enumerate}

\usetikzlibrary{external}
\usetikzlibrary{trees}

\newtheorem{thm}{Theorem}
\newtheorem{lemma}[thm]{Lemma}

\newtheorem{prop}[thm]{Proposition}
\newtheorem{conjecture}[thm]{Conjecture}
\newtheorem{question}{Question}
\newtheorem{claim}[thm]{Claim}
\newtheorem{case}{Case }[thm]
\newtheorem{subcase}{Subcase}[case]

\usepackage[english]{babel}
\usepackage{amssymb}
\usepackage{amsthm}
\usepackage{mathtools}

\usepackage[foot]{amsaddr}
\usepackage{amsmath}
\usepackage{a4wide}
\usepackage{nccmath}
\usepackage{esvect}
\usepackage{hyperref}
\usepackage{cleveref}
\usepackage{subfiles}
\hypersetup{
    colorlinks,
    linkcolor=red,
    citecolor=black,
    urlcolor=blue
}
\usepackage{verbatim}
\usepackage{t1enc}
\usepackage{color}
\usepackage[margin=0.5in]{geometry}
\usepackage{graphicx}
\usepackage{color,soul}
\usepackage[parfill]{parskip}   
\usepackage[numbers,sort&compress]{natbib}
\usepackage{tikz}
\usepackage{tkz-graph}
\usepackage{amsmath}
\usepackage{caption}
\usepackage{subcaption}
\usepackage{mathtools}

\usetikzlibrary{graphs}
\usetikzlibrary{graphs.standard}

\makeatletter
\def\subsubsection{\@startsection{subsubsection}{3}%
  \z@{.5\linespacing\@plus.7\linespacing}{.1\linespacing}%
  {\normalfont\itshape}}
\makeatother

\begin{document}

\author{Rebekah Herrman}
\address[Rebekah Herrman]{Department of Mathematical Sciences, The University of Memphis, Memphis, TN}
\email[Rebekah Herrman]{rherrman@memphis.edu}

\author{Peter van Hintum}
\address[Peter van Hintum]{Department of Pure Maths and Mathematical Statistics, University of Cambridge, UK}
\email[Peter van Hintum]{pllv2@cam.ac.uk}

\title[$(t,r)$ broadcast domination in the infinite grid] 
{$(t,r)$ broadcast domination in the infinite grid}

\linespread{1.3}
\pagestyle{plain}

\begin{abstract}
The $(t,r)$ broadcast domination number of a graph $G$, $\gamma_{t,r}(G)$, is a generalization of the domination number of a graph. $\gamma_{t,r}(G)$ is the minimal number of towers needed, placed on vertices of $G$, each transmitting a signal of strength $t$ which decays linearly, such that every vertex receives a total amount of at least $r$ signal. In this paper we prove a conjecture by Drews, Harris, and Randolph \cite{drews2019} about the minimal density of towers in $\mathbb{Z}^2$ that provide a $(t,3)$ domination broadcast for $t>17$ and explore generalizations. Additionally, we determine the $(t,r)$ broadcast domination number of powers of paths, $P_n^{(k)}$
and powers of cycles, $C_n^{(k)}$.
\end{abstract}
\maketitle 

\section{Introduction}

Let $G= (V(G), E(G))$ be a graph with vertices $V(G)$ and edges $E(G)$. The \emph{domination number} of a graph $G$ is the cardinality of the smallest dominating set of the graph, which is the smallest set $S$ such that every vertex in $V(G)\setminus S$ is adjacent to a vertex of $S$.

In 2014, Blessing, Insko, Johnson, and Mauretour generalized this notion to $(t,r)$ \emph{broadcast domination} \cite{blessing2015}. In broadcast domination, there is a collection of vertices called towers, $\mathcal{T}$, that transmit a signal $t\in\mathbb{N}$ in the following manner. If $u \in \mathcal{T}$, and $v \in G$, then the signal at $v$ from $u$ is denoted $f_u(v)$ and is $f_u(v) = max \{0, t-d(u,v)\}$, where $d(u,v)$ is the distance between $u$ and $v$. The set $\mathcal{T}$ is said to be \emph{$(t,r)$ broadcast dominating} if each tower transmits a signal $t$ and for all $ v \in G$, $\Sigma_{u \in \mathcal{T}} f_u(v) \geq r$. The $(t,r)$ broadcast domination number of $G$, $\gamma_{t,r}(G)$, is the minimum cardinality of a $(t,r)$ broadcasting set  $\mathcal{T}$.  

\par
The $(t,r)$ broadcasting domination number has been studied for two-dimensional grids, paths, triangular grids, matchstick graphs, and $n$-dimensional grids \cite{blessing2015, crepeau2019, drews2019, harris2018, shlomi2019}. Asymptotic bounds of the $(t,2)$ broadcast domination number on finite grids has been studied \cite{randolph2018} as well.  

To describe the $(t,r)$ broadcast domination number of $\mathbb{Z}^2$, we consider the \emph{density} of a set $\mathcal{T}\subset \mathbb{Z}^2$ defined as $\limsup_{n\to \infty} \frac{|\mathcal{T}\cap [-n,n]^2|}{(2n+1)^2}$. Accordingly, $\delta_{t,r}(\mathbb{Z}^2)$ is the minimal density of a $(t,r)$ broadcasting set in $\mathbb{Z}^2$. In 2019, Drews, Harris, and Randolph \cite{drews2019} showed that $\delta_{t,3}(\mathbb{Z}^2) \leq \delta_{t-1,1}(\mathbb{Z}^2) = \frac{1}{2t^2-2t+1}$ for grid graphs $\mathbb{Z}^2$ and conjectured $\delta_{t,3}(\mathbb{Z}^2) = \delta_{t-1,1}(\mathbb{Z}^2)$ for $t>2$. We prove this conjecture for $t>17$.

\begin{thm}\label{densityof3broadcast}
For $t> 17$, $ \delta_{t,3}(\mathbb{Z}^2) = \delta_{t-1,1}(\mathbb{Z}^2)$
\end{thm}
Following the proof of Theorem \ref{densityof3broadcast}, in Section \ref{gensection}, we explore other statements in this direction and suggest some conjectures.
\par

Additionally, we extend the previous result on the $(t,r)$-broadcast domination number of paths \cite{crepeau2019} to powers of paths:

\begin{thm}\label{powerofpath}
Let $n \geq 1$ and $t \geq r \geq 1$. Then  $\gamma_{t,r}(P_n^{(k)}) = \lceil \frac{n+k(r-1)}{2kt-k(r+1)+1} \rceil$.
\end{thm}

Crepeau et. al. found $\gamma_{t,r}(C_n) \leq \lceil \frac{n+r-1}{2t-r} \rceil $ and asked if this bound could be improved \cite{crepeau2019}. We answer their question by giving the exact value for the $(t,r)$ broadcast domination number for all powers of cycles:  
 \begin{thm}\label{powerofcycle}
Let $n\geq 1$ and $t \geq r \geq 1$. Then
$$\gamma_{t,r}(C_n^{(k)})=\begin{cases}
1 &\text{if } n\leq 2(t-r)k+1\\
2 &\text{if } 2(t-r)k+1 < n\leq (2t-r-1)k+1 \\
\left\lceil \frac{n}{(2t-r-1)k+1}\right\rceil &\text{if }n>(2t-r-1)k+1
\end{cases}$$
\end{thm}

 \section{Proof of Theorem \ref{densityof3broadcast}}

First consider the following $(t,1)$ broadcasting set of vertices with minimal density  $\mathcal{T}_0=\{ma+nb:m,n\in\mathbb{Z}\}$ where $a=(t-1,t-2)$ and $b=(t-2,1-t)$. Part of this configuration is shown in Figure \ref{fig:startingconfiguration}.
 
 \begin{figure}[h]
\centering
\includegraphics[scale=0.6]{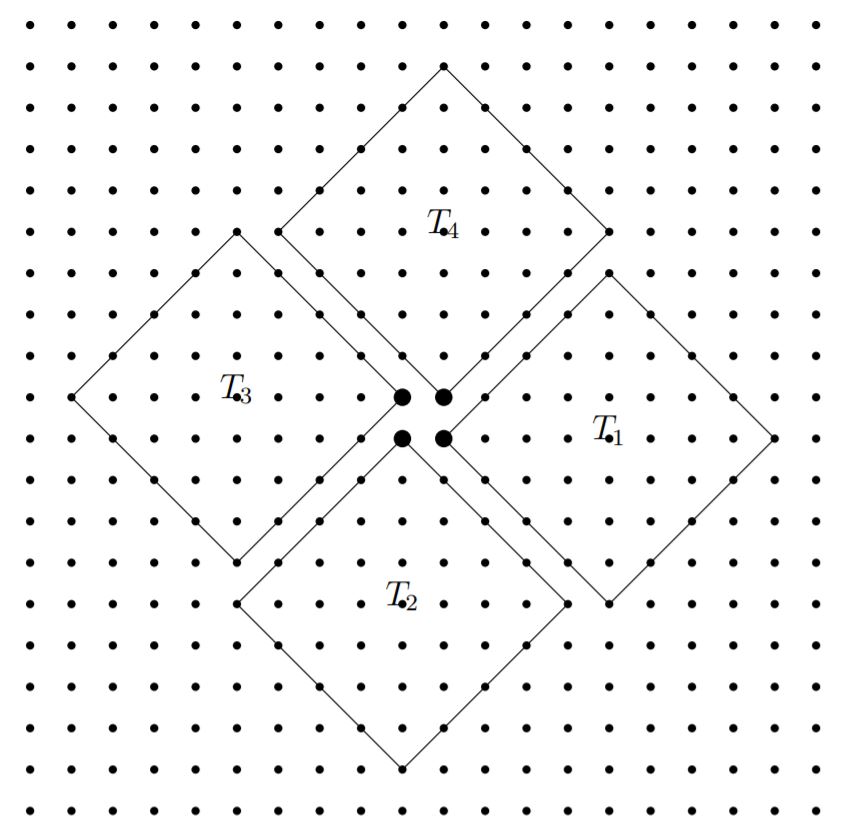}
\caption{An example of a $(5,1)$ broadcasting set. When considered as a $(6,3)$ broadcasting set, the four large vertices in the middle receive excess signal.}
\label{fig:startingconfiguration}
\end{figure}

We consider for every tower the usable transmission which is the sum over the amount transmitted to all the vertices, not exceeding $r$. For a tower at vertex $v$ that is $signal(v):=\sum_{u: d(u,v)\leq t-1} \min\{r,t-d(u,v)\}$. 

Note that the previously described $\mathcal{T}_0$ is also a configuration that provides a $(t+1,3)$  broadcast. We find that four vertices within distance $t-2$ of any tower receive signal 4 rather than the required 3. In Figure \ref{fig:startingconfiguration}, the bold vertices are the one with extra signal. To formalise the notion of extra signal, let $excess(v):=signal(v)-r$ be the \emph{excess signal} received by a vertex $v$ in a given $(t,r)$-broadcasting set of towers.
We would like to attribute the amount of excess to a given tower $T$. Note that the average attributable excess exactly determines the broadcast domination number on vertex transitive graphs.

Our goal is to show $\delta_{t,3}(\mathbb{Z}^2) \geq \delta_{t-1,1}(\mathbb{Z}^2)$. In the starting configuration, we have  exactly $4$ excess attributed to each tower. We want to show that the excess attributed to each tower must be at least $4$ in any $(t+1,3)$ broadcasting configuration, so that the configuration $\mathcal{T}_0$ minimises the excess. 

Henceforth fix some $(t,3)$ broadcasting set of towers. We will prove the following lemma.
\begin{lemma} \label{excesslem}
For any tower at $(x,y)$, there is at least four excess within the vertices $(x,y)+[t-4,t+2]\times[-4,4]$.
\end{lemma}

\begin{proof}
Without loss of generality consider a tower $T$, that will be fixed throughout the argument, at $(-t+2,0)$. We shall consider the following three main cases, along with their subcases. Figures that help visualize the cases are found in the Appendix.

\begin{case}\label{closetowers}
There is another tower $T'$ with $|T'|_1\leq t-2$.
\end{case} 

\begin{subcase}
 $T'$ is not on the $x$-axis.
\end{subcase}

Without loss of generality assume $T'$ is above the $x$-axis, then $T'$ is closer to $(0,1)$ than to $(0,0)$, so $t-|T'-(0,1)|_1\geq 3$ and similarly $t-|T'-(-1,1)|_1\geq 2$ and $t-|T'-(-1,0)|_1\geq 1$. Hence, we find that the excess on $(0,0),(0,1),(-1,0)$ and $(-1,1)$ alone is already more than four, as seen in Figure \ref{fig:81case1}.
\begin{figure}
\centering
\includegraphics[scale=0.6]{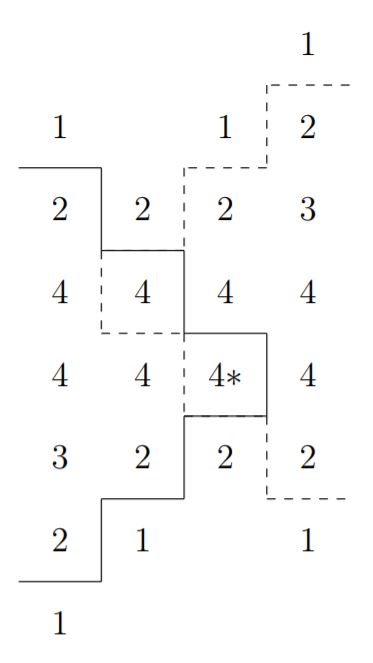}
\caption{The signal received from $T$ and $T'$ in Case 4.1.1, where second tower $T'$ is located at $(t-3,1)$. The line (dashed line resp.) denote the boundary of those vertices receiving at least 2 signal from $T$ ($T''$ resp.). For a minimal $(t-1,1)$ broadcasting set, these regions partition the plane. The $*$ marks the origin.}
\label{fig:81case1}
\end{figure}

\begin{subcase}
 $T'$ is at $(x,0)$ for $x\leq t-3$.
\end{subcase}
 If $T'$ is at $(x,0)$ for $x\leq t-3$, the vertices $(-1,0)$ and $(0,0)$ both have excess at least 2, as seen in Figure \ref{fig:81case2}. 
 
 \begin{figure}
\centering
\includegraphics[scale=0.6]{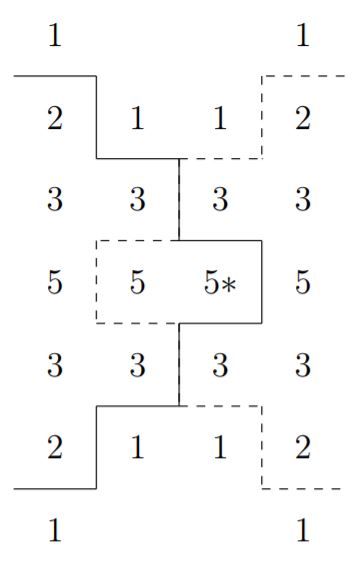}
\caption{The signal received from $T$ and $T'$ in Case 4.1.2, where second tower $T'=(t-3,0)$.}
\label{fig:81case2}
\end{figure}
\begin{figure}
\centering
\includegraphics[scale=0.6]{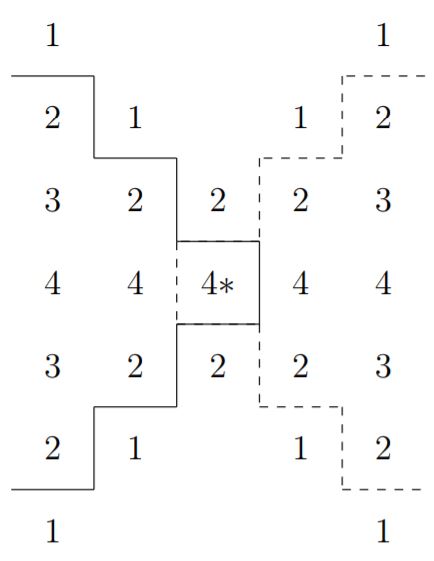}
\caption{The signal received from $T$ and $T'$ in Case 4.1.3, where the second tower is at $T'(t-2,0)$.}
\label{fig:81case3}
\end{figure}
\begin{subcase}
$T'$ is at $(t-2,0)$.
\end{subcase}
 Note that $(-1,0), (0,0)$ and $(1,0)$ all receive at least one excess from $T$ and $T'$ combined.  $(-1,1), (0,1)$ and $(1,1)$ receive 2 signal from $T$ and $T'$ combined, so they need another tower to supply at least one signal. If this is the same tower for two of these, one must must get excess signal. On the other hand consider they receive one signal from three different towers. Either $(-2,1), (2,1)$ or $(0,0)$ must receive excess signal from these towers, or $(0,2)$ receives at least  signal 4 from the three towers combined, as seen in Figure \ref{fig:81case3}

This concludes Case 4.1.

We now distinguish two possible configurations for the tower $T'$ giving additional signal to vertex $(0,0)$. Note that this tower has distance exactly $t-1$ to the origin. Consider whether $T'\in\{(0,t-1),(1,t-2),(0,1-t),(1,2-t)\}$ or not. Note that up to reflection, if $T'\in\{(0,t-1),(1,t-2),(0,1-t),(1,2-t)\}$, we are in the realm of Figure \ref{fig:82}.

\begin{figure}
\centering
\includegraphics[scale=0.6]{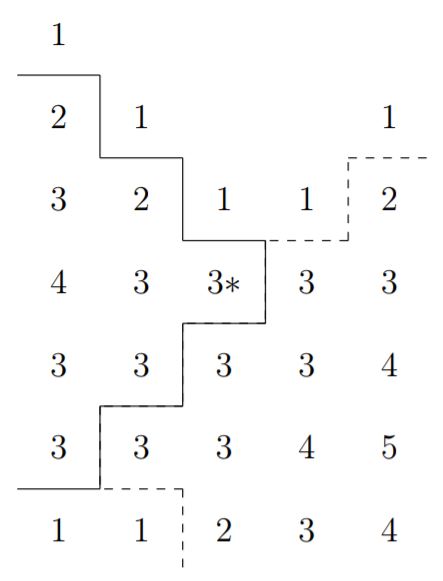}
\caption{The signal received from $T$ and $T'$ in Case 4.2 for the specific example $t=5$ where second tower $T'=(2,4)$.}
\label{fig:82}
\end{figure}

\begin{case}
$T'\not\in\{(0,t-1),(1,t-2),(0,1-t),(1,2-t)\}$
\end{case}
Reflecting if necessary, assume $T'$ is somewhere on $y=x-(t-1)$. 

Note that in this case both $(0,1)$ and $(1,1)$ receive 1 signal from $T$ and $T'$ combined. Hence, they both need signal from an additional tower. 
\begin{subcase}
One additional tower covers both $(0,1)$ and $(1,1)$.
\end{subcase}
This tower will transmit at least a combined signal of three to $(0,0)$ and $(1,0)$, causing a total excess of at least 4 on these four vertices combined.

\begin{subcase} \label{11next}
 $(0,1)$ and $(1,1)$ receive additional signal from two distinct towers.
\end{subcase}
Consider the tower $T''$ giving additional signal to $(-1,1)$. If that tower gives signal at least 2 to $(-2,1)$ or $(-1,0)$, we immediately find the excess. As we additionally know there is no tower at $(0,t-1)$, we find that it must be at $(-1,t)$.

Note that more specifically we know that $(0,1)$ must receive signal from two additional towers. A tower that gives signal 1 to $(0,1)$ must give at least 1 signal to one of $(-1,2)$ and $(1,0)$ and to one of $(-2,3)$ and $(2,-1)$. All of those points already receive 3 signal, so the two additional towers for $(0,1)$ give rise to at least 4 excess on these vertices.

\begin{figure}
\centering
\includegraphics[scale=0.6]{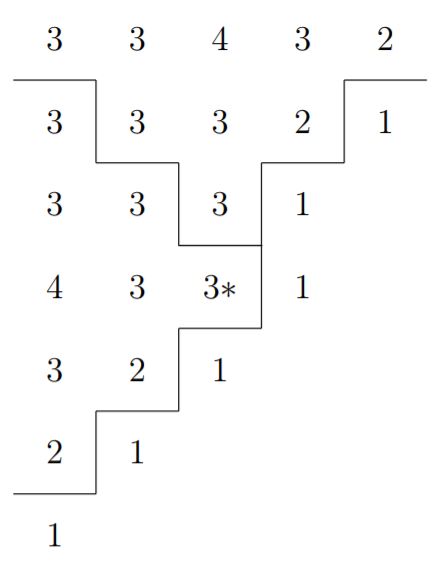}
\caption{The signal received from $T$ and $T'$ in Case 4.3, where second tower $T'=(1,t-2)$.}
\label{fig:83}
\end{figure}
\begin{case}
$T'\in\{(0,t-1),(1,t-2),(0,1-t),(1,2-t)\}$
\end{case}
Without loss of generality $T'=(1,2-t)$. Note that $(-1,1)$ receives only signal 2 from $T$ and $T'$, so receives additional signal from another tower $T''$. By Case \ref{closetowers}, we only need to consider towers at distance $t-1$ from $(-1,1)$. There are only two significant cases. If $T''$ has $x$-coordinate at least 1, then the excess signal on $(0,0),(1,0)$ and $(1,-1)$ is at least 4 already. Hence, $T''$ is either $(0,t-1)$ or $T(-1,t)$. 
\begin{subcase}
$T''=(0,t-1)$
\end{subcase}
Note that $(1,1)$ and $(1,2)$ only receive 2 signal from towers $T, \;T'$ and $T''$. If these two were reached by the same tower say $T'''$, then one of the two must receive signal 2 from $T'''$. If that is $(1,1)$, note that $(0,0),(0,1),(1,0)$ and $(1,1)$ all receive excess at least 1. If it is $(1,2)$, note that $(0,0),(0,2),(1,2)$ and $(1,3)$ all receive excess at least 1, as seen in Figure \ref{fig:83}.

\begin{subcase}
$T''=(-1,t)$
\end{subcase}
This case is completely analogous to Subcase \ref{11next}.

On the other hand, suppose the points $(1,1)$ and $(1,2)$ receive signal 1 from two distinct towers. If either of these towers transmits 2 signal to $(0,1),(0,2), (1,3)$ or $(1,0)$, the excess is immediately more than 4. The towers transmit 2 to $(1,1)$ and $(1,2)$ respectively, then $(2,2)$ receives 1 excess signal and $(3,1)$ receives 2 excess signal.
\end{proof}
The next goal is to show that for large  $t$, we have excess at least four times the number of towers. 
\begin{lemma}\label{excessconclusion}
Let $t>17$. For any $(t,3)$ broadcasting set $\mathcal{T}$ there is at least $4|\mathcal{T}|$ excess.
\end{lemma} 
\begin{proof}
We devise a way to attribute excess to towers. First to all towers $T$ with no other towers within $T+[-6,6]\times [-8,8]$, assign 4 excess from the rectangle $T+[t-4,t+2]\times[-4,4]$. Note that this excess exists by Lemma \ref{excesslem} and that these rectangles are disjoint.

 Let $R_i$ be a $[-6,6]\times [-8,8]$ rectangle around a tower $T_i$. If a tower $T_j$ lies in $R_i$, place an edge between $T_i$ and $T_j$. Suppose $T'$ lies in the rectangle around $T$, so the edge $TT'$ exists. We find that all the vertices in $R'=\frac{T+T'}{2} +[-4,2]\times[-4,4]$ receive at least 3 excess from $T$ and $T'$. Moreover, $R'$ intersects at most four regions of the form $T''+[t-4,t+2]\times[-4,4]$ with $T''\in \mathcal{T}$ as considered in Lemma \ref{excesslem}. Therefore, at least $6\cdot 8-4\cdot4=32$ excess remains available in $R'$. This is cumulative in the sense that if regions of the form $R'$ overlap for different edges in the graph, then still at least 32 excess is available per edge.
As the number of edges is at least half the number of vertices, we find that for every vertex, at least 16 excess can be assigned to that vertex. Hence, we find at least $4|\mathcal{T}|$ excess.
\end{proof}

We are now ready to prove Theorem \ref{densityof3broadcast}.
\begin{proof}
Let $G_{2n+1,2n+1}$ be the $2n+1$ by $2n+1$ grid. We then need at least $3(2n+1)^2$ signal to be transmitted. By Lemma \ref{excessconclusion}, a $(t,3)$-broadcasting set $\mathcal{T}$ of towers can transmit at most $|\mathcal{T}| 3(t-1)^2$ signal effectively. Therefore $|\mathcal{T}|\geq \frac{3(2n+1)^2}{3(t-1)^2}$, so we find
\begin{align*}
\delta_{t,3}(\mathbb{Z}^2)&\geq \lim_{n\to\infty} \frac{ \left(\frac{3(2n+1)^2}{3(t-1)^2}\right)}{(2n+1)^2}\\
&=\frac{1}{(t-1)^2}\\
&=\delta_{t-1,1}(\mathbb{Z}^2)
\end{align*}
\end{proof}

\section{Generalizations of the $(t,r)$ broadcast number for grids}\label{gensection}

The proof of Theorem \ref{densityof3broadcast} suggests that the result may be extended to any odd value of $r$. Note first the following simple, though seemingly unobserved fact;
\begin{prop}\label{t1bound}
For all $t,k\geq 1$; $$\delta_{t,1}(\mathbb{Z}^2)\geq \delta_{t+k,1+2k}(\mathbb{Z}^2)$$
\end{prop}
\begin{proof}
It suffices to show a $(t,1)$ broadcasting set of towers $\mathcal{T}$ is also $(t+k,1+2k)$ broadcasting. Consider a vertex $v\in\mathbb{Z}^2$. As $\mathcal{T}$ is $(t,1)$-broadcasting, $\exists T\in\mathcal{T}$ with $d(T,v)<t$. Find a vertex $u\in\mathbb{Z}^2$ with $d(T,u)=d(T,v)+d(u,v)=t$, which is possible in the plane. Again, as $\mathcal{T}$ is $(t,1)$ broadcasting, there is a $T'\in\mathcal{T}$ with $d(T',u)<t$. Now note that if all towers transmitted $t+k$ of signal, then $v$ receives $t+k-d(T,v)=k+d(u,v)$ signal from tower $T$ and $t+k-d(T',v)\geq t+k-d(u,v)-d(T',u)\geq k+1-d(u,v)$ from tower $T'$. In total $v$ thus receives signal at least $k+d(u,v)+k+1-d(u,v)=2k+1$. Hence, $\mathcal{T}$ is also $(t+k,1+2k)$ broadcasting.
\end{proof}
Similarly we have
\begin{prop}\label{t2bound}
For all $t,k\geq 1$; $$\delta_{t,2}(\mathbb{Z}^2)\geq \delta_{t+k,2+2k}(\mathbb{Z}^2)$$
\end{prop}
\begin{proof}
As before, consider $\mathcal{T}\subset\mathbb{Z}^2$ to be $(t,2)$-broadcasting and $v\in\mathbb{Z}^2$. We will show that if the towers in $\mathcal{T}$ transmitted $t+k$ signal, then all vertices would receive at least $2+2k$ signal. If there is a $T\in\mathcal{T}$ with $d(T,v)\leq t-2$ the proof of the previous lemma suffices completely analogously. If there is no such $T$, there must be $T,T'\in\mathcal{T}$ with $d(T,v)=d(T',v)=t-1$. That implies that $v$ receives signal $k+1$ from both towers and thus $2k+2$ in total.
\end{proof}

In \cite{blessing2015}, Blessing et al. conjectured that in general this inequality is sharp, i.e. that $\delta_{t+1,r+2}(\mathbb{Z}^2)= \delta_{t,r}(\mathbb{Z}^2)$. However, Drews, Harris, and Randolph in \cite{drews2019}, showed by computing these quantities that, in fact, $\delta_{t+1,r+2}(\mathbb{Z}^2)< \delta_{t,r}(\mathbb{Z}^2)$ for several values of $t$ and $r$. Consequently, they formulated a stronger conjecture on the value of $\delta_{t,r}(\mathbb{Z}^2)$ for $r\leq 10$. We believe the improved bounds suggested in \cite{drews2019} are an artifact of the small values of $t$ used in the simulation run by Drews, Harris, and Randolph, as results for $t\leq 15$ were reported in the paper. We propose the following weakening of the conjecture proposed by Blessing, et al. 
\begin{conjecture}\label{mainconj}
For all $r \geq 2,$ there exists $t_0$  such that for all $t\geq t_0$;
$$\delta_{t+1,r+2}(\mathbb{Z}^2)= \delta_{t,r}(\mathbb{Z}^2).$$
\end{conjecture}

In the hopes of proving this result along the line of the proof of Theorem \ref{densityof3broadcast}, we compute the average amount of excess per tower in an optimally $(t,1)$ broadcasting configuration when viewed as a $(t+k,2k+1)$ broadcasting configuration. The task of showing that one cannot achieve a configuration with a smaller average amount of excess per tower remains open, but a proof along the same lines as Lemma \ref{excesslem} seems reasonable. Our attempts have resulted in impenetrable casework, and more ideas to improve elegance would be needed.

\begin{lemma}\label{conjbound}
Let $t>k$. The average excess per tower in an optimally $(t,1)$ broadcasting configuration when viewed as a $(t+k,2k+1)$ broadcasting configuration is $\frac16 k(k+1)(2k+1)$.
\end{lemma} 

\begin{proof}
Consider four towers around the origin at $T_1=(t-1,0), T_2=(-1,-(t-1)), T_3=(-t,1)$ and $T_4=(0,t)$ and call the square formed by these towers $S$. This configuration provides a $(t,1)$- broadcast. To complete the proof, it suffices to show that the starting configuration also provides a $(t+k, 1+2k)$- broadcast.
 
 We shall divide $S$ into two regions. Let $S'$ be the square with corner vertices $(k-1,-(k-1)), (-k, -(k-1)), (-k,k),$ and $(k-1,k)$, along with all points on the boundary, and in the interior of this region. As $t > k$, $S'$ is contained inside $S$, since $k < t$, $-(k-1) > -(t-1)$, $k-1 < t-1$ and $-k > -t$.
\begin{claim}\label{outsideofsquare}
The vertices inside $S$ that have signal at least $r$ and no excess are the vertices that do not lie in $S'$ and are in $S$.
\end{claim}
\begin{proof}
Consider the regions defined by the lines $x+y=k$, $x+y=-k-1$, $x-y=k$ and $x-y=-k-1$.
Note that by symmetry we need only check that there is no excess above the line $x+y=k$.
Above the line $x+y=k$, no vertex receives any signal from $T_2$ and $T_3$. Consider a vertex $(x,y)$ in this region. If this vertex is above $x-y=k$ or below $x-y=-k-1$, it will receive signal from only one tower. This will be signal at least $2k+1$ but will have no excess as it lies in the broadcast zone of exactly one tower.
Otherwise, this vertex will receive signal $t+k-(t-1-x+y)$ from $T_1$ and $t+k-(x+t-y)$ from $T_4$, which amounts to a total signal of $2k+1$.
\end{proof}

In the proof of the next claim, we find that each vertex in $S'$ has excess and calculate how much. This process shows that each vertex in $S'$ has signal greater than $r$. 
\begin{claim}
 The excess of $S'$ is $\frac{ 1}{6} k(k + 1) (2 k + 1)$.
\end{claim}
\begin{proof}
In fact we note that for every $0\leq i\leq k-1$, a vertex on the intersection between $x+y=i$ and $S'$ receives an excess of $2k-2i-1$. We proceed by induction on $i$. For $i=0$, note that $(0,0)$ receives $(t-(t'-1))+(t-t')+(t-(t'+1))+(t-t')=4k$ signal, which corresponds to $2k-1$ excess. For a vertex $v$ with $i\geq 1$, note that at least one of $v-e_1$ and $v-e_2$ was in the intersection between $S'$ and $x+y=i-1$. Fix one of these to be $v'$. Now the distances to three towers increases, while to one tower it decreases. \\
In particular, if $v=v'+e_1$, then $d(v,T_2)=d(v',T_2)-1$, $d(v,T_3)=d(v',T_3)-1$, $d(v,T_4)=d(v',T_4)-1$, and $d(v,T_1)=d(v',T_1)-1$. On the other hand, if $v=v'+e_2$, then $d(v,T_1)=d(v',T_1)-1$, $d(v,T_3)=d(v',T_3)-1$, $d(v,T_4)=d(v',T_4)-1$, and $d(v,T_2)=d(v',T_2)-1$.\\
Either way the signal received by $v$ is 2 less than by $v'$ finishing the induction.

The number of vertices on the intersection between $S'$ and $x+y=i$ is $i+1$, so we find total excess: $\sum_{i=0}^{k-1}(i+1)(2k-2i-1)=\frac{ 1}{6} k(k + 1) (2 k + 1)$
\end{proof} 

Thus, each vertex on the infinite grid with a tiling of this pattern has signal at least $r$. This concludes the proof of Theorem \ref{conjbound}.
\end{proof}

 \section{Proof of Theorem \ref{powerofpath}}
\begin{proof}
We will consider the power of a path, $G=P_n^{(k)}$ on vertex set $\{0,\dots,n-1\}$ with $v_iv_j$ an edge  if and only if $|i-j|\leq k$.
For the lower bound we consider the potentially useful amount of signal transmitted by a tower. Note that from the signal submitted to a vertex at distance at most $t-r$ from a tower, only $r$ can be used to exceed the signal threshold. Hence, the total amount of potentially useful signal transmitted by a tower is at most $(2k(t-r)+1)r+2k((r-1)+(r-2)+\dots+1)=((2t-r-1)k+1)r$. Moreover, as the vertex $v_0$ receives signal at least $r$, there must be a tower at $v_i$ for some $i\leq (t-r)k$. This tower wastes $k((r-1)+(r-2)+\dots+1)=kr(r-1)/2$ of its potentially useful amount of transmitted signal. Similarly, $v_n$ receives signal at least $r$. We may conclude that the total amount of transmitted signal needed is at least $nr+kr(r-1)$. This gives the lower bound $\left \lceil\frac{n+k(r-1)}{(2t-r-1)k+1} \right\rceil$.

For the upper bound consider $\mathcal{T}=\{v_i:0\leq i\leq n-1, i\equiv (t-r)k\mod(2t-r-1)k+1\}$ if $(n-1)\mod (2t-r-1)k+1$ is between $(t-r)k$ and $2(t-r)k+1$. Otherwise, let $\mathcal{T}=\{v_i:0\leq i\leq n-1, i\equiv (t-r)k\mod(2t-r-1)k+1\}\cup \{v_{n-1}\}$.

Note that vertices $v_i$ with $i\leq (t-r)k$ all receive enough signal from the tower at $v_{(t-r)k}$. By construction, the last tower is at distance at most $(t-r)$ away from the vertex $v_{n-1}$, so all the vertices not between two towers receive enough signal.

Now consider a vertex $v_i$ between two towers, say $i=l((2t-r-1)k+1)+(t-r)k+p$ where $0\leq p< (2t-r-1)k+1$ and both $v_{l((2t-r-1)k+1)+(t-r)k}$ and $v_{\min\{n,(l+1)((2t-r-1)k+1)+(t-r)k}\}$ are in $\mathcal{T}$. Then 
\begin{align*}
d(v_i,v_{l((2t-r-1)k+1)+(t-r)k})&+d(v_i,v_{\min\{(l+1)((2t-r-1)k+1)+(t-r)k,n\}})\\
&\leq \left\lceil\frac{p}{k}\right\rceil + \left\lceil \frac{(2t-r-1)k+1-p}{k}\right\rceil\\
&=(2t-r-1)+\left\lceil\frac{p}{k}\right\rceil + \left\lceil \frac{1-p}{k}\right\rceil\\
&\leq (2t-r-1)+1\\
&=2t-r
\end{align*}
Thus, the broadcast received by vertex $v_i$ is 
\begin{align*}
\max\{t-&d(v_i,v_{l((2t-r-1)k+1)+(t-r)k}),0\}+\max\{t-d(v_i,v_{\min\{(l+1)((2t-r-1)k+1)+(t-r)k,n\}}),0\}\\
&\geq 2t-(d(v_i,v_{l((2t-r-1)k+1)+(t-r)k}+d(v_i,v_{\min\{(l+1)((2t-r-1)k+1)+(t-r)k,n\}}))\\
&\geq 2t-(2t-r)=r
\end{align*}
Thence, all vertices receive sufficient signal.

\end{proof}
When $k=1$, we are left with a path, and obtain $\gamma_{t,r} (P_n) = \left\lceil \frac{n+r-1}{2t-r} \right\rceil $, agreeing with the result by Crepeau, et al.

\section{Proof of Theorem \ref{powerofcycle}}
\begin{proof}
If $n\leq 2(t-r)k+1$, then any vertex is at most distance $(t-r)$ from any other vertex, so a tower at any vertex is $(t,r)$-broadcasting. If, on the other hand, $n> 2(t-r)k+1$ we find that for all $0\leq i<n$, $d(v_i, v_{i+(t-r)k+1})=(t-r)+1$. Hence, no one tower can be $(t,r)$-broadcasting. 
For $n\leq (2t-r-1)k+1$, $\mathcal{T}=\{0,\left\lfloor\frac{n}{2}\right\rfloor\}$ is $(t,r)$-broadcasting.

First we will show the upper bound. When $2(t-r)k+1 < n$, consider the set $\mathcal{T}=\{v_i:i\equiv 0 \mod (2t-r-1)k+1\}\cap \{v_0,\dots,v_n\}$. Evidently, $|\mathcal{T}|=\left\lceil \frac{n}{(2t-r-1)k+1}\right\rceil$. Moreover, we will show that these towers are $(t,r)$-broadcasting. Consider vertex $v_i$. Choose $l$ and $p$ such that $p\in\{0,\dots,(2t-r-1)k\}$ and $i=l((2t-r-1)k+1)+p$. Note that the two towers closest to $v_i$ are $v_{l((2t-r-1)k+1)}$ and $v_{\min\{(l+1)((2t-r-1)k+1),n\}}$. We find that the sum of the distance between each tower and $v_i$ is
\begin{align*}
d(v_i,v_{l((2t-r-1)k+1)})+d(v_i,v_{\min\{(l+1)((2t-r-1)k+1),n\}})
&\leq \left\lceil\frac{p}{k}\right\rceil + \left\lceil \frac{(2t-r-1)k+1-p}{k}\right\rceil\\
&=(2t-r-1)+\left\lceil\frac{p}{k}\right\rceil + \left\lceil \frac{1-p}{k}\right\rceil\\
&\leq (2t-r-1)+1\\
&=2t-r
\end{align*}
Thus, the broadcast received by vertex $v_i$ is 
\begin{align*}
\max\{t-&d(v_i,v_{l((2t-r-1)k+1)}),0\}+\max\{t-d(v_i,v_{\min\{(l+1)((2t-r-1)k+1),n\}}),0\}\\
&\geq 2t-(d(v_i,v_{l((2t-r-1)k+1)}+d(v_i,v_{\min\{(l+1)((2t-r-1)k+1),n\}}))\\
&\geq 2t-(2t-r)=r
\end{align*}

Note that from the signal submitted to a vertex at distance at most $t-r$ from a tower, only $r$ is used to exceed the signal threshold. Hence, the total amount of potentially useful signal submitted by a tower is at most $(2k(t-r)+1)r+2k((r-1)+(r-2)+\dots+1)=((2t-r-1)k+1)r$. The total signal needed to saturate all the vertices is at least $nr$. Hence, $\gamma_{t,r}(C_n^{(k)})\geq \left\lceil \frac{nr}{r((2t-r-1)k+1)}\right\rceil = \left\lceil \frac{n}{(2t-r-1)k+1}\right\rceil$.
\end{proof}

\section{Concluding Remarks}

A natural next direction would be to consider $n$-dimensional generalizations. Analogously to the 2 dimensional definitions, let the density of a set $\mathcal{T}\subset\mathbb{Z}^n$ be defined to be $\limsup_{m\to \infty} \frac{|\mathcal{T}\cap [-m,m]^n|}{(2m+1)^n}$ and let $\delta_{t,r}(\mathbb{Z}^n)$ be the minimal density of a $(t,r)$ broadcasting set $\mathcal{T}\subset\mathbb{Z}^n$.

\begin{question}
	Is there a relationship between $\delta_{t,r}(\mathbb{Z}^n)$ and $\delta_{t-1,r-2}(\mathbb{Z}^n)$ for some $t$, and $r$? 
\end{question}

In complete parallel to Propositions \ref{t1bound} and \ref{t2bound}, we have that $\delta_{t+k,1+2k}(\mathbb{Z}^n)\leq \delta_{t,1}(\mathbb{Z}^n)$ and $\delta_{t+k,2+2k}(\mathbb{Z}^n)\leq \delta_{t,2}(\mathbb{Z}^n)$ by an analogous proof. Note that in dimensions $n>2$, unlike in dimensions one and two, $l_1$-balls of constant radius do not partition $\mathbb{Z}^n$, so even the exact value of $\delta_{t,1}(\mathbb{Z}^n)$ can be hard to obtain. In 3 dimensions this amounts to efficiently covering space with octahedrons.
\par

In another direction, the continuous generalization of Conjecture \ref{mainconj} might provide a lot of insight. We say a set of towers $\mathcal{T}\subset\mathbb{R}^2$ is $(t,r)$ broadcasting if all points in points $v\in\mathbb{R}^2$ satisfy that
$$\sum_{T\in\mathcal{T}} \max\{t-d(T,v),0\} \geq r$$
where $d$ is some metric on $\mathbb{R}^2$. It is natural to look for the minimal density $\limsup_{x\to \infty} \frac{\text{card} (\mathcal{T}\cap [-x,x]^2)}{4x^2}$ of a $(t,r)$-broadcasting set. For $d$ the Euclidean $\ell_2$ distance, this problem is intimately related to efficient sphere packing. To stay as close to the discrete context as possible, let $d$ be the $\ell_1$ distance. Let $\delta'_{t,r}(\mathbb{R}^2)$ be the smallest density of a $(t,r)$ broadcasting set in $\mathbb{R}^2$. Note that in this definition being $(t,r)$ broadcasting and being $(1,\frac{r}{t})$ broadcasting are equivalent. In fact for $\alpha>0$, $\delta'_{t,r}(\mathbb{R}^2)=\delta'_{\alpha t,\alpha r}(\mathbb{R}^2)$. Analogously to Conjecture \ref{mainconj}, we believe
\begin{conjecture}\label{contconj}
There exists $ \gamma_0>0$ such that for all $\gamma\leq \gamma_0$,
$$\delta'_{1,\gamma}(\mathbb{R}^2)=\lim_{\epsilon\to 0}\delta'_{1-\gamma/2,\epsilon}(\mathbb{R}^2)=\frac{1}{4(1-\frac{\gamma}{2})^2}$$
\end{conjecture}
The right equality follows from the fact that the set $\mathcal{T}_\epsilon=\{ma+nb:m,n\in\mathbb{Z}\}$ with $a=(1-\frac{\gamma}{2}-\epsilon,1-\frac{\gamma}{2}-\epsilon)$ and $b=(1-\frac{\gamma}{2}-\epsilon,\frac{\gamma}{2}+\epsilon-1)$ is $(1-\gamma/2, \epsilon)$ broadcasting and has asymptotic density $\frac{1}{4(1-\frac{\gamma}{2}-\epsilon)^2}$, which tends to $\frac{1}{4(1-\frac{\gamma}{2})^2}$ as $\epsilon\to 0$. Moreover, the set $\mathcal{T}_0$ immediately shows $\delta'_{1,\gamma}\leq \frac{1}{4(1-\frac{\gamma}{2})^2}$.

When viewing the discrete setting as an approximation of the continuous setting, Conjecture \ref{contconj} would indicate that the minimal $t_0$ as a function of $r$ in Conjecture \ref{mainconj} would be at most linear, i.e. $t_0=O(r)$.

\section*{Acknowledgements}
The authors would like to thank their supervisor professor B\'ela Bollob\'as for his continuous support and for comments on an earlier draft of this paper.

\bibliographystyle{abbrv}
\bibliography{references}

\end{document}